\documentclass[11pt]{amsart}

\usepackage{amssymb, latexsym, amsfonts, pigpen, enumitem, mathrsfs, amsthm, bbm, stackrel, longtable, hhline}
\usepackage[matrix,arrow,curve]{xy}
\usepackage[colorlinks=true,pagebackref=true,citecolor=cyan,linkcolor=magenta]{hyperref}
\usepackage[table]{xcolor}
\hypersetup{linktocpage}

\usepackage[letterpaper,top=1.05in, bottom=1.05in, left=1.05in, right=1.05in]{geometry}
\usepackage[width=.65\textwidth]{caption}

\newtheorem{lemma}[equation]{Lemma}
\newtheorem{proposition}[equation]{Proposition}

\newtheorem{theoremintr}{Theorem}

\newtheorem{conjectureintr}{Conjecture}

\theoremstyle{definition}
\newtheorem{definition}[equation]{Definition}
\newtheorem{example}[equation]{Example}
\newtheorem{remark}[equation]{Remark}
\newtheorem{notation}[equation]{Notation}

\numberwithin{equation}{section}

\newcommand{\CH}{\mathrm{CH}}
\newcommand{\OGr}{\mathrm{OGr}}
\newcommand{\Gr}{\mathrm{Gr}}
\newcommand{\struct}{\mathcal{O}}
\newcommand{\taut}{\mathcal{T}}
\newcommand{\id}{\mathrm{id}}
\newcommand{\Pic}{\mathrm{Pic}}
\newcommand{\Tor}{\mathrm{Tor}}
\newcommand{\kernel}{\mathrm{Ker}}
\newcommand{\coker}{\mathrm{Coker}}
\newcommand{\im}{\mathrm{Im}}
\newcommand{\res}{\mathrm{res}}
\newcommand{\CK}{\mathrm{CK}}
\newcommand{\BPtr}{\mathrm{BP}\langle 2\rangle}

\makeatletter

\begin{document}

\title[Counter-examples to a conjecture of Karpenko via truncated BP-cohomology]{Counter-examples to a conjecture of Karpenko via truncated~Brown-Peterson cohomology}

\author{Victor Petrov}
\address{St. Petersburg State University, 14th Line V.O. 29b, 199178 Saint Petersburg, Russia and PDMI RAS, Fontanka 27, 191023 Saint Petersburg, Russia}
\email{\href{mailto:victorapetrov@googlemail.com}{victorapetrov@googlemail.com}}

\author{Alois Wohlschlager}
\address{Ludwig-Maximilians-Universit\"at M\"unchen, Theresienstr. 39, D-80333 M\"unchen, Germany}
\email{\href{mailto:wohlschlager@math.lmu.de}{wohlschlager@math.lmu.de}}

\author{Egor Zolotarev}
\address{Ludwig-Maximilians-Universit\"at M\"unchen, Theresienstr. 39, D-80333 M\"unchen, Germany}
\email{\href{mailto:zolotarev@math.lmu.de}{zolotarev@math.lmu.de}, \href{mailto:zolotarev-egv@yandex.ru}{zolotarev-egv@yandex.ru}}


\begin{abstract}
    Let $G$ be a split semisimple linear algebraic group and let $X$ denote the generically twisted variety of Borel subgroups in $G$. Nikita Karpenko conjectured that the map from the Chow ring of $X$ to the associated graded ring of the topological filtration on the Grothendieck ring of $X$ is an isomorphism. After having been verified for many $G$, the conjecture was disproved by Nobuaki Yagita for some spinor groups. Later, other counter-examples were constructed by Baek--Karpenko and Baek--Devyatov.
    We present a new method for constructing counter-examples that is based on the connection of the truncated Brown--Peterson cohomology with the connective K-theory.
    Using this method, we disprove the conjecture for new groups, including $\mathrm{Spin}_{15}$, which is now the smallest known spinor group for which the conjecture fails.
\end{abstract}

\setcounter{tocdepth}{1}

\maketitle

\tableofcontents
\allowdisplaybreaks
\section{Introduction}
    The Chow ring $\CH^*(X)$ of a smooth variety $X$ is an important invariant of cohomological nature that contains a lot of geometric and arithmetic information about $X$. However, this ring is difficult to compute in general. In order to obtain a good approximation of the Chow ring, Alexander Grothendieck introduced the topological filtration on the zeroth K-group of $X$ \cite{SGA6} and constructed a canonical epimorphism
    $$ \varphi_X\colon\CH^*(X)\twoheadrightarrow \mathrm{gr}_{\tau} \mathrm{K}_0(X) $$
    that sends the class of a subvariety $Z$ to the class of its structure sheaf $[\mathcal{O}_Z]$, where $\mathrm{gr}_\tau$ denotes the associated graded ring for the topological filtration. Unfortunately, this epimorphism is not injective in general.
    One can ask whether this epimorphism is an isomorphism for some class of varieties. An interesting class of varieties for which the Chow rings have a lot of applications to problems in algebra is that of projective homogeneous varieties. While for the whole class the above morphism is clearly not an isomorphism, motivated by a number of examples Nikita Karpenko conjectured that a positive answer to this question holds for generically twisted varieties of complete flags.
    \begin{conjectureintr}[{\cite[Conjecture 1.1]{Karpenko_types_GFE}}]\label{karp_conj}
        Let $G$ be a split semisimple linear algebraic group over a field $F$, and let $B\subset G$ be a Borel subgroup. Fix a generic $G$-torsor $E$ (see \cite[Chapter I, \S5]{GMS_coh_inv}). Then the morphism $\varphi_{E/B}$ is an isomorphism.
        \end{conjectureintr}
    In \cite[Lemma 2.1]{Karpenko_types_GFE}, it is shown that the conjecture is independent of the choice of the generic $G$-torsor $E$.
    We also note that the conjecture is equivalent to the same statement with the Borel subgroup replaced by any special parabolic subgroup $P\subset G$ \cite[Lemma 4.2]{Karpenko_types_AC}. In particular, a positive answer to the conjecture for $G$ leads to a computation of the Chow ring $\CH^*(E/P)$ for any such $P$, since the Grothendieck ring $\mathrm{K}_0(E/P)$ is known from the work of Panin \cite{Panin_E/P}, and for such varieties, the topological filtration coincides with the computable $\gamma$-filtration \cite[Corollary 7.4]{Karpenko_topological=gamma}. 
    
    The conjecture has been verified for the classical groups of type $\mathrm{A}_n$ and $\mathrm{C}_n$ in \cite[Theorem 1.1]{Karpenko_types_AC}, for the special orthogonal groups $\mathrm{SO}_n$ in \cite[Theorem 1]{Victor_SO}, and for the exceptional groups $\mathrm{G}_2$, $\mathrm{F}_4$, $\mathrm{E}_6^{sc}$ in \cite[Propositions 5.1 and 5.2]{Karpenko_types_GFE}. For the spinor groups $\mathrm{Spin}_{2n+1}$, the conjecture was verified for $n\leq 5$ in \cite[Theorem 3.1]{Karpenko_Spin_11_12} (note that it suffices to consider only odd spinor groups, since the conjecture for $\mathrm{Spin}_{2n+1}$ is equivalent to the same conjecture for $\mathrm{Spin}_{2n+2}$ by \cite[Proposition 2.16]{Baek_Karpenko_counter-example}) and later disproved for $n=8$ and $n=9$ by Nobuaki Yagita \cite[Theorem 10.4 and Lemma 10.5]{YagSpin}. Subsequently, a simplified argument was presented by Karpenko for $n=8$ \cite[Theorem 1.1]{karpenko-counterexample}, and later extended by Baek--Karpenko to $n=9,10$ \cite[Theorem 1.3]{Baek_Karpenko_counter-example}, and by Baek--Devyatov to $n$ any $2$-power $\geq 8$ \cite[Theorem 1.3]{Baek_Dev_counter-examples}. In this paper we find new counter-examples (and present yet another argument for some previous ones) among spinor groups. Namely, we prove the following:
    \setcounter{theoremintr}{1}
    \begin{theoremintr}\label{main_theorem}
        Let $F$ be a field of characteristic zero. Then Conjecture \ref{karp_conj} is false for both $\mathrm{Spin}_{2n+1}$ and $\mathrm{Spin}_{2n+2}$ for $n=7,8,9,10,13,14,15,16$ (the cases $n=8,9,10,16$ were previously known).
    \end{theoremintr}
    In particular, we obtain the smallest known counter-example, given by  $\mathrm{Spin}_{15}$. The only remaining smaller odd spinor case still open is $\mathrm{Spin}_{13}$. We suspect that the boundary for the conjecture is somehow related to the fact that $\mathrm{Spin}_{13}$ is the largest odd spinor group acting on the projectivization of the spin representation with an open orbit. Consequently, we expect the conjecture to hold for $\mathrm{Spin}_{13}$, which would make our counter-example the smallest possible among all spinor groups. We also emphasize that the counter-example for $\mathrm{Spin}_{15}$ presented here is the most involved and it comes from torsion in $\CH_4$, see Remark \ref{remark_degree_chow}. In contrast, all other known counter-examples -- both those from the literature and the new ones presented in this paper -- are based on torsion in $\CH_3$.
\subsection{Overview of the method}
    According to \cite[Theorem 3.1]{KarpCK}, the conjecture is equivalent to the injectivity of the restriction map for the connective K-theory of $E/P$, where $P\subset G$ is a special parabolic subgroup (see \S \ref{section_2:2} for the recollection on the connective K-theory). Using this reformulation we prove the following (see Proposition \ref{counter1} for a precise statement): if $A^*$ is a free oriented cohomology theory that specializes to the connective K-theory and if Conjecture \ref{karp_conj} holds for $G$, then we have
    $$
    \Tor_1^{A^*(F)}(\underline{A}^*(E/P),\mathbb{Z}[\beta])=0,
    $$
    where $\underline{A}^*(E/P)$ is the cokernel of the restriction map for $E/P$ (see Definition \ref{def:rat_and_irrat}). This approach can also be applied $p$-locally for a prime $p$. Taking $A^*$ to be the truncated Brown--Peterson cohomology $\BPtr^*$, which has the coefficient ring $\mathbb{Z}_{(p)}[v_1,v_2]$, the Tor group is isomorphic to the $v_2$-torsion subgroup of $\underline{\BPtr}^*(E/P)$, see Example \ref{example:bptr}. Consequently, to disprove the conjecture for $G$, it suffices to find a non-rational element in $\BPtr^*$-cohomology of $G/P$ whose multiple by $v_2$ is rational. The advantage of this method, in contrast to the previous ones, is that all computations are performed in the cohomology of the split form rather than in the generically twisted version. 
\subsection{Computations of rational elements}\label{subsection_code}
    In the case of $G=\mathrm{Spin}_{2n+1}$ we perform explicit computations, working with the special parabolic subgroup $P_n\subset\mathrm{Spin}_{2n+1}$ (the numeration of simple roots follows Bourbaki). We begin by constructing the multiplicative generators and relations of the algebraic cobordism ring of the split maximal orthogonal Grassmannian $\OGr(n)=\mathrm{Spin}_{2n+1}/P_n$ (see Proposition \ref{prop:generators}). To make the computation easier, we work with an approximation of the relations in the truncated Brown--Peterson cohomology $\BPtr^*(\OGr(n))$ (Proposition \ref{prop:relations_z_i}). An explicit set of generators of the rational elements is given in Lemma \ref{lemma:rat_generators}.

    Due to their structure, reduction modulo the simplified relations is straightforward (applying them repeatedly until all monomials are square-free).
    We have written computer programs using Rust and the Singular computer algebra system for this task, enabling us to calculate the rational elements of suitable degree in $\BPtr^*(\OGr(n))$ very explicitly.
    They have been published as open-source under \url{https://codeberg.org/alois3264/karpenkos-conjecture} (mirrored at \url{https://github.com/alois31/karpenkos-conjecture}).
    Due to the long running time of the programs in the most interesting cases, their output is also available in the results branch.

    Analyzing the results of the computations, 
    we find the desired $v_2$-torsion elements for $n=7,8,9,10,13$ in Proposition \ref{prop:desired_v2_tors}. This proves Theorem \ref{main_theorem} for these values of $n$. For the remaining cases, we provide an induction argument on $n$ that applies only when the torsion index $t(\mathrm{Spin}_{2n+1})$ as computed by Totaro \cite{totaro} increases strictly.

\subsection{Acknowledgements}
    Part of this work was originally contained in the Master's thesis of the third author, written under the supervision of the first author at SPbU. We are deeply grateful to Nikita Geldhauser, who then suggested to the second author to construct the counter-examples using a computer, which led to this paper. We also would like to thank Alexey Ananyevskiy, Nikita Geldhauser, Andrei Lavrenov, and Maksim Zhykhovich for valuable discussions.

    The first author was supported by the Foundation for the Advancement of Theoretical Physics and Mathematics “BASIS”. The second author has worked on this subject during their (ongoing) doctoral studies, which are partially funded by the Studienstiftung des deutschen Volkes.
    The work of the third author is supported by the DFG research grant AN 1545/4-1.
\section{Preliminaries}
In this section we collect preliminaries on oriented cohomology theories in the sense of Levine--Morel that we use in the main part of the text. Throughout the paper we work over a field $F$ of characteristic zero.
\subsection{Oriented cohomology theories}

An \textit{oriented cohomology theory} in the sense of Levine--Morel (see \cite[Definition 1.1.2]{LevMor}) consists of a contravariant functor $A^*$ from the category of smooth $F$-varieties to the category of commutative, graded rings together with pushforward maps for projective morphisms. Namely, for each projective morphism of smooth varieties $f\colon Y\to X$ of relative dimension $d$, there exists a homomorphism of abelian groups $f_*\colon A^*(Y)\to A^{*+d}(X)$. These data are required to satisfy certain axioms, such as the projection formula, homotopy invariance, and the projective bundle formula. For a morphism of smooth $F$-varieties $f\colon X\to Y$ the associated map of rings $A^*(Y)\to A^*(X)$ is called \textit{pullback along} $f$ and denoted by $f^*$. By a standard abuse of notation, we say that $A^*$ is an oriented cohomology theory, omitting the chosen system of pushforward maps.

We say that the oriented cohomology theory $A^*$ satisfies the \textit{localization axiom} if for any smooth $F$-variety $X$ with smooth closed subvariety $i\colon Z\hookrightarrow X$ and open complement $j\colon U\hookrightarrow X$, the sequence 
$$ A^{*-d}(Z)\xrightarrow{i_*} A^*(X)\xrightarrow{j^*} A^*(U)\to 0 $$
is exact. Note that some authors extend $A^*$ to non-smooth varieties and require exactness of the localization sequence for any subscheme $Z$ (see e.g., \cite[Definition 2.1]{Vishik_op_omega}), but the above form is enough for our purposes.

Let $A^*$ be an oriented cohomology theory. Then one may use the Grothendieck method \cite{Grothendieck_Chern} to define Chern classes $c_i^A(E)\in A^i(X)$ of a vector bundle $E\to X$. Moreover, one can construct a $\mathbb{Z}$-graded formal group law $F_A$ over $A^*(F)$ that satisfies
$$ c_1^A(L\otimes M)=F_A(c_1^A(L),c_1^A(M))\in A^1(X) $$
for any pair of line bundles $L,M$ on $X$, see \cite[Lemma 1.1.3]{LevMor}. This formal group law contains important information about the theory $A^*$.

\subsection{Algebraic cobordism, Brown--Peterson cohomology, and connective K-theory}\label{section_2:2}

Recall that Levine and Morel constructed a universal oriented cohomology theory $\Omega^*$ \cite{LevMor}, known as \textit{algebraic cobordism}. This theory satisfies the localization axiom and the formal group law associated with $\Omega^*$ is the universal one. In particular, the coefficient ring $\Omega^*(F)$ is isomorphic to the Lazard ring $\mathbb{L}\cong\mathbb{Z}[a_1,a_2,\dots]$ with $\mathrm{deg}(a_i)=-i$. 

Given a $\mathbb{Z}$-graded (homogeneous of degree $1$) formal group law defined via a ring morphism $\mathbb{L}\to R$, we can consider the \emph{free} oriented cohomology theory with this law, defined by
$$ A^*(-):=\Omega^*(-)\otimes_{\mathbb{L}}R$$
with obvious pushforward maps. Any such theory satisfies the localization axiom. One particularly useful example that we need is the \textit{Brown–Peterson cohomology} $\mathrm{BP}^*(-)$ with respect to a prime $p$. This is a free oriented cohomology theory whose formal group law is the universal $p$-typical one (see \cite[Theorem A2.1.25]{Ra04}) over
$$ \mathrm{BP}^*(F)\cong\mathbb{Z}_{(p)}[v_1,v_2,\dots]$$
with $\mathrm{deg}(v_k)=1-p^k$. The formal group law of $\mathrm{BP}^*$ can be described in terms of its logarithm, see Appendix \ref{appendix_fgl}. For a natural number $n$, we also consider the \textit{truncated Brown--Peterson cohomology} $\mathrm{BP}\langle n\rangle^*(-)$ defined as the quotient of $\mathrm{BP}^*(-)$ by the ideal generated by $v_k$ for $k>n$. This theory is also a free oriented cohomology theory, and its coefficient ring is isomorphic to $\mathbb{Z}_{(p)}[v_1,\dots,v_n]$.

Another oriented theory that we use is the \textit{connective K-theory} $\CK^*$ \cite{Cai_CK}. Let $X$ be a smooth $F$-variety. In the notation of \textit{loc.cit.}, $\CK^i(X)$ is the group $\CK^{i,-i}(X)$ that is defined by
$$\CK^i(X):=\im(\mathrm{K}_0(\mathcal{M}_i(X))\to \mathrm{K}_0(\mathcal{M}_{i-1}(X))),$$
where $\mathcal{M}_i(X)$ is the category of coherent sheaves on $X$ with codimension of support $\geq i$. The formal group law associated with $\CK^*$ is the multiplicative one over the polynomial ring
$$\CK^*(F)\cong\mathbb{Z}[\beta]$$
with generator $\beta$ of degree $-1$. The connective K-theory is free by \cite[Remark 5.6.2 and Theorem 6.3]{Dai_Levine_CK}, i.e., the canonical morphism $\Omega^*\to\CK^*$ induces an isomorphism
$$ \Omega^*(-)\otimes_{\mathbb{L}}\mathbb{Z}[\beta]\cong \CK^*(-) $$
of oriented theories over $F$. Finally, we notice that the Artin--Hasse exponent provides an isomorphism between the $p$-local multiplicative formal group law and $F_{\mathrm{BP}\langle 1\rangle}$. In particular, for $p=2$ we have an isomorphism of presheaves of graded rings $\mathrm{BP}\langle 1\rangle^*\cong\CK^*\otimes\mathbb{Z}_{(2)}$ by \cite[Theorem 6.9]{Vishik_op_omega} (note that the pushforwards are different), while for an odd prime $p$ the presheaf $\CK^*\otimes\mathbb{Z}_{(p)}$ splits into a direct sum of $p-1$ shifted copies of $\mathrm{BP}\langle 1\rangle^*$.

\subsection{Restriction map and (ir)rational elements}
Let $A^*$ be a free oriented cohomology theory over $F$ and let $L/F$ be a field extension. Then $A^*(-\times_F L)$ is a new oriented theory equipped with the natural transformation
$$ \res^A_{L/F}\colon A^*(-)\to A^*(-\times_F L) $$
of cohomology theories over $F$. We call this morphism the \textit{restriction of scalars}, or simply the \textit{restriction}. Fix an algebraic closure $\overline{F}$ of $F$.
\begin{definition}\label{def:rat_and_irrat}
    Let $X$ be a smooth $F$-variety. We denote by $\overline{A}^*(X)$ the image of the restriction of scalars map along the extension $\overline{F}/F$:
    $$ \overline{A}^*(X):=\im(A^*(X)\xrightarrow{\res^A_{\overline{F}/F}}A^*(X_{\overline{F}})), $$
    and called it the \textit{subring of rational elements}. 
    We also denote by $\underline{A}^*(X)$ the cokernel of the same (restriction) morphism of $A^*(F)$-modules:
    $$ \underline{A}^*(X):=\coker(A^*(X)\xrightarrow{\res^A_{\overline{F}/F}} A^*(X_{\overline{F}})), $$
    and called it the \textit{$A^*(F)$-module of irrational elements}.
\end{definition}

\begin{remark}
    In \cite[Definition 1]{Andrey_Victor_irrational} the module of irrational elements is defined as
    $$A^*(X_{\overline{F}})\otimes_{A^*(X)}A^*(F)$$
    and these definitions are \textit{not} equivalent. Indeed, the above tensor product is isomorphic to the quotient by the ideal generated by the image of the restriction of $\res^A_{\overline{F}/F}$ to the augmented ideal $A^*(X)^+$, in other words the cokernel of the (restriction) morphism of \emph{graded rings}. In particular, it is an $A^*(F)$-algebra, while our module of irrational elements does not admit any natural ring structure.
\end{remark}

\section{Consequences of the Karpenko conjecture}
In this section we prove some consequences of the Karpenko conjecture that we use to find new counter-examples. Throughout this section $G$ is a split semisimple linear algebraic group over a field $k$ of characteristic $0$ and $P\subset G$ is a special parabolic subgroup. We also fix a generic $G$-torsor $E$ (see \cite[Chapter I, \S 5]{GMS_coh_inv}), which is defined over an extension $F/k$. 
By abuse of notation, we write $G/P$ for $(E/P)_{\overline{F}}\cong (G/P)_{\overline{F}}$.
\begin{proposition}\label{counter1}
Let $A^*$ be a free oriented cohomology theory with a morphism of presheaves of rings to the connective K-theory $A^*\to \CK^*$ such that for any smooth $F$-variety $X$ we have $A^*(X)\otimes_{A^*(F)}\mathbb{Z}[\beta]\cong \CK^*(X).$ If Conjecture \ref{karp_conj} holds for $G$, then
$$
\Tor_1^{A^*(F)}(\underline{A}^*(E/P),\mathbb{Z}[\beta])=0.
$$
\end{proposition}
\begin{proof}
Consider the restriction map $\res^A_{\overline{F}/F}\colon A^*(E/P)\to A^*(G/P)$. It gives us obvious short exact sequences of $A^*(F)$-modules
$$
0\to \kernel(\res^A_{\overline{F}/F})\to A^*(E/P)\to \overline{A}^*(E/P)\to 0,
$$
$$
0\to \overline{A}^*(E/P)\to A^*(G/P)\to \underline{A}^*(E/P)\to 0.
$$
Applying $-\otimes_{A^*(F)} \mathbb{Z}[\beta]$ and using the fact that $A^*(G/P)$ is a free $A^*(F)$-module, we obtain the following exact sequences
$$
\kernel(\res^A_{\overline{F}/F})\otimes_{A^*(F)} \mathbb{Z}[\beta]\to \CK^*(E/P)\to \overline{A}^*(E/P)\otimes_{A^*(F)} \mathbb{Z}[\beta]\to 0,
$$
$$
0\to \Tor_1^{A^*(F)}(\underline{A}^*(E/P),\mathbb{Z}[\beta])\to \overline{A}^*(E/P)\otimes_{A^*(F)} \mathbb{Z}[\beta]\to \CK^*(G/P).
$$
It is not hard to see that the restriction map $\res^{\CK}_{\overline{F}/F}$ decomposes into the composition
$$\CK^*(E/P)\to \overline{A}^*(E/P)\otimes_{A^*(F)} \mathbb{Z}[\beta]\to \CK^*(G/P)$$
and it is injective by Conjecture \ref{karp_conj}, see \cite[Theorem 3.1]{KarpCK}. Thus, the first map is injective, but it is also surjective by the first exact sequence above. We obtain that the second morphism is injective as well and the desired vanishing follows from the second exact sequence above.
\end{proof}

\begin{remark}
    An obvious example of an oriented theory that satisfies assumption of the previous proposition is algebraic cobordism $\Omega^*$. However, computing the group
    $$ \Tor_1^{\mathbb{L}}(\underline{\Omega}^*(E/P),\mathbb{Z}[\beta]) $$
    seems to be complicated in general.

    Yagita conjectured that the restriction map $\Omega^*(E/P)\to \Omega^*(G/P)$ is injective \cite[Conjecture 8.6]{YagSpin} (he states it only for Brown--Peterson cohomology of generalized generic Rost motives, but one can easily deduce the above form using \cite{J-inv,Sem_Zhyk} and Quillen's idempotents). To the best of our knowledge this conjecture is still open. We claim that the following conditions are equivalent:
    \begin{center}
        \begin{enumerate}
            \item The Yagita conjecture holds for $G$ and $\Tor_1^{\mathbb{L}}(\underline{\Omega}^*(E/P),\mathbb{Z}[\beta])$ vanishes.
            \item The Karpenko conjecture holds for $G$.
        \end{enumerate}
    \end{center}
    Assume that (1) holds. The injectivity of $\res^\Omega_{\overline{F}/F}$ implies existence of the short exact sequence of $\mathbb{L}$-modules
    $$ 0\to \Omega^*(E/P)\xrightarrow{\res^\Omega_{\overline{F}/F}} \Omega^*(G/P)\to \underline{\Omega}^*(E/P)\to 0, $$
    which induces the exact sequence
    \begin{align}\label{ex_seq_yagita} 0\to \Tor_1^{\mathbb{L}}(\underline{\Omega}^*(E/P),\mathbb{Z}[\beta])\to \CK^*(E/P)\xrightarrow{\res^{\CK}_{\overline{F}/F}} \CK^*(G/P). \end{align}
    The left term vanishes by the assumption. Hence, the restriction map $\res^{\CK}_{\overline{F}/F}$ is injective, which is equivalent to the Karpenko conjecture.
    
    Conversely, the Karpenko conjecture implies Yagita's conjecture \cite[Lemma 8.7]{YagSpin} and the claim (1) follows using exact sequence \ref{ex_seq_yagita} again.
\end{remark}

Since some computations that we will present are much easier locally at a prime, mainly due to the coefficient ring of the Brown-Peterson ring having more favourable degrees of the parameters in comparison to algebraic cobordism, we need the following $p$-local version of the above proposition.

\begin{proposition}\label{prop:p_local_karpenko}
    Let $p$ be a prime and let $A^*$ be a free oriented cohomology theory with a morphism of presheaves of rings to the $p$-local connective K-theory $A^*\to\CK^*\otimes\mathbb{Z}_{(p)}$ such that for any smooth $X$ we have $A^*(X)\otimes_{A^*(F)}\mathbb{Z}_{(p)}[\beta]\cong \CK^*(X)\otimes\mathbb{Z}_{(p)}.$ If Conjecture \ref{karp_conj} holds for $G$, then we have
    $$ \Tor_1^{A^*(F)}(\underline{A}^*(E/P),\mathbb{Z}_{(p)}[\beta])=0. $$
\end{proposition}
\begin{proof}
    The argument is the same as in the previous proposition.
\end{proof}

\begin{remark}
    One can deduce from the main results of \cite{J-inv,Sem_Zhyk} that the $p$-torsion of $\CH^*(E/P)$ can be non-trivial only for torsion primes of $G$, and by the proof of \cite[Theorem 3.1]{KarpCK} the same holds for $\CK^*(E/P)$. Therefore, the previous proposition is interesting only if $p$ is a torsion prime of $G$.
    In fact, using Proposition \ref{prop:t_mult_rat} one sees directly that if $p$ is not a torsion prime, then already $\underline{A}^*(E/P)=0$.
\end{remark}

\begin{example}\label{example:bptr}
    The main example for us is $A^*=\mathrm{BP}\langle2\rangle^*$ for $p=2$, whose coefficient ring is given by $\mathbb{Z}_{(2)}[v_1,v_2]$. In this case the projection $\mathrm{BP}\langle2\rangle^*(F)\to\CK^*(F)\otimes\mathbb{Z}_{(2)}=\mathbb{Z}_{(2)}[\beta]$ is the quotient by the principal ideal generated by $v_2$ (element $v_1$ goes to $\beta$). It follows that the group $$\Tor^{\mathbb{Z}_{(2)}[v_1,v_2]}_1(\underline{\mathrm{BP}\langle2\rangle}^*(E/P),\mathbb{Z}_{(2)}[\beta])$$ is isomorphic to the group of $v_2$-torsion irrational elements
    $$ {}_{v_2}\underline{\mathrm{BP}\langle2\rangle}^*(E/P)= \{ \alpha\in\underline{\mathrm{BP}\langle2\rangle}^*(E/P)\,|\, v_2\cdot\alpha=0 \}. $$
    In explicit terms, in order to construct a non-trivial element in the above group, it is enough to find a non-rational element in $\mathrm{BP}\langle2\rangle^*(G/P)$ whose $v_2$-multiple is rational.
\end{example}

\section{Algebraic cobordism of split maximal orthogonal Grassmannians}
In this section we construct multiplicative generators for the algebraic cobordism ring of a split maximal orthogonal Grassmannian. The construction is a slight modification of the one known for Chow groups \cite[\S 86]{EKM} (the original reference is \cite[\S 2]{Vishik_OGr}). We need the following technical lemma.

\begin{lemma}[cf., {\cite[Proposition 58.10]{EKM}}]\label{element}
Let $q:E\to X$ be a vector bundle and let $E'\subset E$ be a subbundle of corank $r$. Then there exists an element $\alpha\in \Omega^0(\mathbb{P}(E'))$, such that
\begin{equation}\label{equation1}
    j_*(\alpha)=\sum_{k=0}^r (-1)^{k} c_k^{\Omega}(E/E')\xi^{r-k},
\end{equation}
where $\xi$ denotes the first Chern class of $\mathcal{O}(1)$ over $\mathbb{P}(E)$, and $j$ is the embedding of the corresponding projective bundles $j:\mathbb{P}(E')\to \mathbb{P}(E)$.
\end{lemma}
\begin{proof}
Using the projective bundle formula for $E/E'$ we have 
$$
    \sum\limits_{k=0}^r (-1)^k c_k^{\Omega}(E/E')\Tilde{\xi}^{r-k}=0,
$$
in $\Omega^*(\mathbb{P}(E/E'))$, where $\Tilde{\xi}$ is the first Chern class of the canonical line bundle over $\mathbb{P}(E/E')$. Applying the pullback morphism with respect to the canonical map $\mathbb{P}(E)\setminus \mathbb{P}(E')\to \mathbb{P}(E/E')$, we see that the restriction of the right hand side of the formula \ref{equation1} to $\mathbb{P}(E)\setminus \mathbb{P}(E')$ is trivial. The result follows from the exactness of the localization sequence
$$
   \Omega^0(\mathbb{P}(E'))\xrightarrow{j_*} \Omega^r(\mathbb{P}(E))\to \Omega^r(\mathbb{P}(E)\setminus \mathbb{P}(E'))\to 0.
$$
\end{proof}

\begin{notation}
Let $V$ be a vector space of dimension $2n+1$ over a field $F$ with a quadratic form $q:V\to F$. Consider the variety of maximal totally isotropic subspaces in $V$, which is a closed subscheme of the usual Grassmannian $\Gr(n,V)$. We call this variety the \textit{maximal orthogonal Grassmannian of} $q$ and denote it by $\OGr(q)$. If $q$ is a split quadratic form $q=q_s$, then we write simply $\OGr(n)$. This variety can be identified with the quotient of $\mathrm{Spin}_{2n+1}$ by the parabolic subgroup $P_n$. We define the \textit{tautological} vector bundle $\taut_n$ of rank $n$ over $\OGr(n)$ to be the pullback of the tautological bundle along $\OGr(n)\hookrightarrow \Gr(n,V)$. There is an embedding $j_n:\OGr(n-1)\hookrightarrow \OGr(n)$ given by the choice of an isotropic line $L\subset V$.
\end{notation}
\begin{notation}
We denote the projective quadric correspond to the split quadratic form $q_s$ by $Q$. Choose a maximal totally isotropic subspace $U$ of $V$ and denote by $i:\mathbb{P}(U)\to Q$ the embedding of its projectivization in the quadric. We denote by $l:=i_*(1)\in \Omega^{n}(Q)$ the corresponding class in the algebraic cobordism ring of $Q$ and by $h=c_1^{\Omega}(\mathcal{O}_Q(1))\in\Omega^1(Q)$ the first Chern class of the canonical bundle.
\end{notation}
\begin{proposition}\label{prop:generators}
There are unique elements $z_k\in \Omega^k(\OGr(n))$ for $1\leqslant k\leqslant n$, such that the monomials $z_1^{i_1}z_2^{i_2}\dots z_n^{i_n}$ for $i_k\in \{0,1\}$ form a free basis of $\Omega^*(\OGr(n))$ over $\mathbb{L}$, and the Chern classes of $\taut_n^\vee$ are computed as follows (where $F_\Omega(t,t)=:\sum_{i=1}^\infty d_it^i$):
$$
    c_k^\Omega(\taut_n^\vee)=(-1)^k\sum\limits_{i=0}^{n-k} d_{i+1}z_{k+i}.
$$
\end{proposition}
\begin{proof}
Since $\OGr(n)=\mathrm{Spin}_{2n+1}/P_n$, it follows from the Bruhat decomposition that the Chow motive of $\OGr(n)$ splits into a direct sum of Tate summands. By \cite[Corollary 2.8]{Vishik-Yagita} the $\Omega^*$-motive also splits. Hence, $\Omega^*(\OGr(n))$ is a free $\mathbb{L}$-module of rank $2^n$ and we have the K\"unneth isomorphism: 
$$
\Omega^*(Q)\otimes_{\mathbb{L}} \Omega^*(\OGr(n))\xrightarrow{\simeq}\Omega^*(Q\times \OGr(n))
$$
given by $x\otimes y\mapsto x\times y$. Thus, the ring $\Omega^*(Q\times \OGr(n))$ is a free $\Omega^*(\OGr(n))$-module with basis $\{h^i\times 1, lh^i\times 1\}$ for $0\le i\le n-1$ by \cite[\S 4.1]{Morava_SO}. Applying Lemma \ref{element} to $\taut_n\subset V\mathbbm{1}=V\times\OGr(n)$ we obtain an element $\alpha\in \Omega^0(\mathbb{P}(\taut_n))$ that satisfies formula \ref{equation1}. Now we take the pushforward of $\alpha$ along the closed embedding $k:\mathbb{P}(\taut_n)\hookrightarrow Q\times \OGr(n)$. It gives us a decomposition
$$
k_*(\alpha)=\sum\limits_{i=1}^n h^{n-i}\times z_i +\sum\limits_{i=0}^{n-1} lh^i\times y_i. 
$$
We claim that these $z_i$'s are our desired elements. First we prove that monomials in $z_i$'s are free generators of $\Omega^*(\OGr(n))$. Consider the canonical morphism of oriented cohomology theories $\mathrm{pr}_{\CH}\colon\Omega^*\to \CH^*$, which coincides with the quotient by $\mathbb{L}^{<0}\Omega^*$. We have $\mathrm{pr}_{\CH}(z_k)=e_k$ in the notation of \cite[pp. 357--358]{EKM}. Hence, by the corresponding fact for the Chow ring \cite[Theorem 86.12]{EKM} and the graded Nakayama lemma \cite[Lemma 5.5]{Morava_SO} the monomials $z_1^{i_1}z_2^{i_2}\dots z_n^{i_n}$ for $i_k\in \{0,1\}$ form a free basis of $\Omega^*(\OGr(n))$ over $\mathbb{L}$.

For the computation of $c_k^\Omega(\taut_n^\vee)$ in terms of our basis we denote by $m:Q\hookrightarrow \mathbb{P}(V)$ the obvious embedding and take the pushforward of $k_*(\alpha)$ along $m\times \id_{\OGr(n)}$. Using the projection formula, we obtain
$$
    (m\times \id_{\OGr(n)})_*(k_*(\alpha))=\sum\limits_{i=1}^n m_*(1)z_i\xi^{n-i}+\sum\limits_{i=0}^{n-1}m_*(l)y_i\xi^i,
$$
where $\xi$ denotes the first Chern class of $\struct(1)$ over $\mathbb{P}(V)$. Since $m_*(l)=m_*(i_*(1))$ is the class of $(n-1)$-dimension projective subspace in $\mathbb{P}(V)$, it follows that $m_*(l)=\xi^{n+1}$. By normalization identity (see \cite[\S 2.5]{Pan09}), we have $m_*(1)=c_1^\Omega(\mathcal{O}(2))=F_\Omega(\xi,\xi)=\sum_{i\geqslant 1} d_i\xi^i$. Observe that 
$$(m\times \id_{\OGr(n)})_*\circ k_*(\alpha)=((m\times \id_{\OGr(n)})\circ k)_*(\alpha)=j_*(\alpha),$$
where $j:\mathbb{P}(\taut_n)\to \mathbb{P}(V)\times \OGr(n)$ is the embedding of the projective subbundle. On the other hand, $\alpha$ was chosen so that $j_*(\alpha)$ satisfies formula \ref{equation1}. Summing up all of the above we obtain the following formula
$$
\sum\limits_{k=0}^{n+1} (-1)^kc_k^\Omega(V\mathbbm{1}/\taut_n)\xi^{n+1-k}=\sum\limits_{i=1}^n \sum\limits_{j\geqslant 1} d_j z_i\xi^{n+j-i} + \sum\limits_{i=0}^{n-1} y_i\xi^{n+1+i}.
$$
Using the projective bundle formula, we can compare coefficients to see that
$$c_k^\Omega(V\mathbbm{1}/\taut_n)=(-1)^k\sum_{i=0}^{n-k}d_{i+1}z_{k+i}$$
for $1\leq k\leq n$ and $c_{n+1}^\Omega(V\mathbbm{1}/\taut_n)=0$. It remains to show that the Chern classes of $V\mathbbm{1}/\taut_n$ and $\taut_n^\vee$ coincide. Consider the short exact sequence of vector bundles
$$
0\to V\mathbbm{1}/\taut_n^\perp\to V\mathbbm{1}/\taut_n\to \taut_n^\perp/\taut_n\to 0.
$$
By duality $V\mathbbm{1}/\taut_n^\perp\cong \taut_n^\vee$. The line bundle $\taut_n^\perp/\taut_n$ carries a nondegenerate quadratic form, hence is isomorphic to its dual. Since $\Pic(\OGr(n))=\CH^1(\OGr(n))$ is torsion-free, we conclude that $\taut_n^\perp/\taut_n\cong \struct_{\OGr(n)}$. Therefore, we have equalities of the total Chern classes
$c^\Omega(\taut_n^\vee)=c^\Omega(V\mathbbm{1}/\taut_n^\perp)=c^\Omega(V\mathbbm{1}/\taut_n)$.

To prove uniqueness, assume that there are elements $z_k'\in \Omega^i(\OGr(n))$ that satisfy the same conditions. We want to prove that $z_k=z_k'$ for $1\leq k\leq n$. Use descending induction on $k$. Suppose that $z_k=z_k'$ for $k>k_0$ (for $k_0=n$ this condition is trivially satisfied). Then $$(-1)^{k_0}(2z_{k_0}+d_2z_{k_0+1}+\dots+d_{n+1-k_0}z_n)=c_{k_0}^\Omega(\taut_n^\vee)=(-1)^{k_0}(2z_{k_0}'+d_{k_0+1}z_2'+\dots+d_{n+1-k_0}z_n').$$ The last sum is equal to $(-1)^{k_0}(2z_{k_0}'+d_{k_0+1}z_2+\dots+d_{n+1-k_0}z_n)$ by induction hypothesis. Hence, we obtain $2z_{k_0}=2z_{k_0}'$ and the result follows since $\Omega^*(\OGr(n))$ is torsion-free.
\end{proof}
\begin{remark}
Notice that $\mathrm{pr}_{\CH}(c_k^\Omega(\taut_n^\vee))=c_k^{\CH}(\taut_n^\vee)=(-1)^k2e_k$ differs by a sign from \cite[Proposition 86.13]{EKM}. The reason is the different conventions about the definition of Chern classes.
\end{remark}
The multiplicative generators constructed in the theorem are compatible with the embeddings of the split maximal orthogonal Grassmannians in the following sense.
\begin{lemma}\label{pullback1}
$$
j_n^*(z_k)=
    \begin{cases}
      z_k, & \text{if}\ k< n \\
      0, & \text{it}\ k=n
    \end{cases}
$$
\end{lemma}
\begin{proof}
The proof is similar to the proof of the uniqueness in the previous proposition. Use descending induction on $k$. Suppose that $j_n^*(z_k)=z_k$ for $k>k_0$ (here we set $z_n\in\Omega^*(\OGr(n-1))$ to be zero). Then $j_n^*(c_{k_0}^\Omega(\taut_n^\vee))$ is equal to $$j_n^*((-1)^{k_0}(2z_{k_0}+d_{k_0+1}z_2+\dots+d_{n+1-k_0}z_n))=(-1)^{k_0}(2j_n^*(z_{k_0})+d_{k_0+1}z_2+\dots+d_{n+1-k_0}z_n).$$ On the other hand, since $j_n^*(\taut_n^\vee)\cong \taut_{n-1}^\vee\oplus \struct_{\OGr(n)}$, we have
$$j_n^*(c_{k_0}^\Omega(\taut_n^\vee))=c_{k_0}^\Omega(\taut_{n-1}^\vee)=(-1)^{i_0}(2z_{k_0}+d_{k_0+1}z_2+\dots+d_{n-k_0}z_{n-1}).$$
It follows that $2j_n^*(z_{k_0})=2z_{k_0}$. We are done since $\Omega^*(\OGr(n-1))$ is torsion-free.
\end{proof}

\begin{lemma}\label{lemma:chern_relations}
    The Chern classes of the tautological bundle $\mathcal{T}_n$ and its dual satisfy the following relations for $1\le k\le n$:
    $$ \sum_{i=0}^{2k} c_{i}^\Omega(\taut_n)c_{2k-i}^\Omega(\taut_n^{\vee})=0. $$
\end{lemma}
\begin{proof}
    We have an equality of total Chern classes $c^\Omega(\taut_n^\vee)=c^\Omega(V\mathbbm{1}/\taut_n)$ by the proof of Proposition \ref{prop:generators}. The desired relations follows from the Whitney formula. 
\end{proof}
\begin{remark}
    In the Chow ring $\CH^*(\OGr(n))$, all relations can be obtained by dividing the above relations among the Chern classes by $4$. For algebraic cobordism, the same strategy theoretically works, but we do not have a closed formula for the Chern classes of the dual vector bundle.
\end{remark}

\section{Approximate relations in Brown--Peterson cohomology}
In principle the results of the previous section give an explicit computation of $\BPtr^*(\OGr(n))$ (or indeed $A^*(\OGr(n))$ for any free cohomology theory $A^*$) in terms of the multiplicative generators $z_i$ and relations between them.
In fact we have done this explicitly with the help of a computer for $n\le7$.
For larger $n$, it is useful to have to some approximations, that will be presented in the following.

\begin{lemma}
In $\mathrm{BP}^*(\OGr(n))$ (and consequently, also in $\BPtr^*(\OGr(n))$) the following relations hold for all $1\le k\le n$:
\begin{align*}
\begin{split}
    (-1)^{k+1}(c^*_k)^2&=\sum_{i=0}^{k-1}(-1)^i2c^*_{i,2k-i}-v_1\sum_{i=0}^k(-1)^iP_{1,0}(k-i)c^*_{i,2k+1-i}+v_1^2\sum_{i=0}^k(-1)^iP_{2,0}(k-i)c^*_{i,2k+2-i}\\
    &-v_1^3\sum_{i=0}^{k+1}(-1)^iP_{3,0}(k-i)c^*_{i,2k+3-i}-v_2\sum_{i=0}^{k+1}(-1)^iP_{0,1}(k-i)c^*_{i,2k+3-i}\\
    &+v_1^4\sum_{i=0}^{k+1}(-1)^iP_{4,0}(k-i)c^*_{i,2k+4-i}+v_1v_2\sum_{i=0}^{k+1}(-1)^iP_{1,1}(k-i)c^*_{i,2k+4-i}+O(v^5).
\end{split}
\end{align*}
The short notations used are $c^*_I=\prod_{i\in I}c_i^{\mathrm{BP}}(\taut_n^\vee)$ with $c^*_0=1$ and $c^*_i=0$ for $i>n$, and $O(v^5)$ denotes some sum of terms divisible by a product of powers of the $v_i$ of degree $\le-5$ (note that each $v_i$ has degree $1-2^i$).
The polynomials appearing in the coefficients are:
\begin{align*}
    P_{1,0}(t)&=2t+1,\\
    P_{2,0}(t)&=t^2+2t+1,\\
    P_{3,0}(t)&=\tfrac{2t^3+9t^2+25t+24}6,\\
    P_{0,1}(t)&=2t+3,\\
    P_{4,0}(t)&=\tfrac{t^4+8t^3+47t^2+124t+108}{12},\\
    P_{1,1}(t)&=2t^2+8t+8.
\end{align*}
\end{lemma}
\begin{proof}
    The proof is a straightforward computation that uses obtained relations among the Chern classes and approximate formulas for the formal inverse power series associated with the formal group law $F_{\mathrm{BP}}$. Here are the details. We put
    \begin{align*}
    \begin{split}
        P&:=1-v_1c^*_1+v_1^2c^*_2-v_1^3(c^*_{1,1,1}-3c^*_{1,2}+4c^*_3)-v_2(c^*_{1,1,1}-3c^*_{1,2}+3c^*_3)\\
        &-v_1^4(c^*_{1,1,1,1}-5c^*_{1,1,2}+4c^*_{2,2}+5c^*_{1,3}-9c^*_4)-v_1v_2(c^*_{1,1,1,1}-5c^*_{1,1,2}+4c^*_{2,2}+5c^*_{1,3}-8c^*_4).
    \end{split}
    \end{align*}
    The chosen value of $P$ has no particular significance, except for being invertible and simplifying the next computation as much as possible. Denote by $x_i$ the Chern roots of the vector bundle $\mathcal{T}_n^\vee$ in $\mathrm{BP}$-cohomology, and by $[-1]_F$ the formal inverse corresponding to the formal group law of $\mathrm{BP}^*$. We also put $c_i=c_i^{\mathrm{BP}}(\taut_n)$.
    Using the approximation of the formal inverse from Appendix \ref{appendix_fgl} we obtain:
    \begin{align*}
        (-1)^ic_iP&=(-1)^i\sigma_i([-1]_F(x_1),\ldots,[-1]_F(x_n))P=R(x_1,\ldots,x_n)\cdot P+O(v^5)
    \end{align*}
    for the symmetric polynomial
    \begin{align*}
        R&:=m_{1^i}+v_1m_{2,1^{i-1}}+v_1^2(m_{3,1^{i-1}}+m_{2^2,1^{i-2}})+v_1^3(2m_{4,1^{i-1}}+m_{3,2,1^{i-2}}+m_{2^3,1^{i-3}})+v_2m_{4,1^{i-1}}\\
        &+v_1^4(4m_{5,1^{i-1}}+2m_{4,2,1^{i-2}}+m_{3^2,1^{i-2}}+m_{3,2^2,1^{i-3}}+m_{2^4,1^{i-4}})+v_1v_2(3m_{5,1^{i-1}}+m_{4,2,1^{i-2}}),
    \end{align*}
    where $m_\lambda$ is the monomial symmetric polynomial associated with an integer partition $\lambda$ (see Appendix \ref{appendix_sym_pol} for the definition). Simplification using the expressions of $m_\lambda$'s in terms of $\sigma_\lambda$'s from Appendix \ref{appendix_sym_pol}, and substituting back $c_i^*=\sigma_i(x_1,\ldots,x_n)$, shows that the element $(-1)^ic_iP$ has the following approximation
    \begin{align*}
    \begin{split}
        c^*_i-v_1(i+1)c^*_{i+1}+v_1^2\tfrac{i^2+3i+2}2c^*_{i+2}-v_1^3(-c^*_{1,i+2}+\tfrac{i^3+6i^2+17i+24}6c^*_{i+3})-v_2(-c^*_{1,i+2}+(i+3)c^*_{i+3})&\\
       +v_1^4(-(i+2)c^*_{1,i+3}+\tfrac{i^4+10i^3+59i^2+194i+216}{24}c^*_{i+4})+v_1v_2(-(i+2)c^*_{1,i+3}+(i^2+6i+8)c^*_{i+4})+O(v^5).&
    \end{split}
    \end{align*}
    After multiplying the $k$-th relation on Chern classes from Lemma \ref{lemma:chern_relations} by $P$ and applying the above formulas for $(-1)^ic_iP$, we obtain the expression
    \begin{align}\label{equation_ck^2_QT} 
    \begin{split}
    (-1)^{k+1}(c_k^*)^2&=Q_{0,0}-v_1Q_{1,0}+v_1^2Q_{2,0}-v_1^3(Q_{3,0}+T_{3,0})
    -v_2(Q_{0,1}+T_{0,1})\\
    &+v_1^4(Q_{4,0}+T_{4,0})+
    v_1v_2(Q_{1,1}+T_{1,1})+O(v^5).
    \end{split}
    \end{align}
    Here $Q_{s,t}$ and $T_{s,t}$ are polynomials in the dual Chern classes, where each $Q_{s,t}$ consists of monomials of two Chern classes and each $T_{s,t}$ consists of monomials of three Chern classes. A straightforward computation shows that $Q_{0,0}$, $Q_{1,0}$, and $Q_{0,1}$ already have the desired form, while the remaining $Q_{s,t}$ are related to the $P_{s,t}$ as follows:
    \begin{align*}
        Q_{2,0}&=\sum_{i=0}^k(-1)^i\left(P_{2,0}(k-i)+\tfrac{(k^2+k)}{2}\cdot 2\right)c^*_{i,2k+2-i}+(-1)^{k-1}\tfrac{k^2+k}2(c^*_{k+1})^2\\
        Q_{3,0}&=\sum_{i=0}^{k+1}(-1)^i\left(P_{3,0}(k-i)+\tfrac{k^2+k}2P_{1,0}(k+1-i)\right)c^*_{i,2k+3-i}\\
        Q_{4,0}&=\sum_{i=0}^{k+1}(-1)^i\left(P_{4,0}(k-i)+\tfrac{k^2+k}{2}P_{2,0}(k+1-i)+\tfrac{k^4+2k^3+23k^2+46k}{24}\cdot 2\right)c^*_{i,2k+4-i}\\
        &+(-1)^{k-2}\tfrac{k^4+2k^3+23k^2+46k}{24}(c^*_{k+2})^2\\
        Q_{1,1}&=\sum_{i=0}^{k+1}(-1)^i\left(P_{1,1}(k-i)+(k^2+2k)\cdot 2\right)c^*_{i,2k+4-i}+(-1)^{k-2}(k^2+2k)(c^*_{k+2})^2.
    \end{align*}
    In turn, the cubic polynomials $T_{s,t}$ are given by the following formulas
    \begin{align*}
        \begin{split}
            T_{3,0}&=T_{0,1}=-c^*_1\sum_{i=0}^{2k+2}(-1)^ic^*_{i,2k+2-i},\\
            T_{4,0}&=T_{1,1}=c^*_1\sum_{i=0}^{2k+3}(-1)^i(i-1)c^*_{i,2k+3-i}.
        \end{split}
    \end{align*}
    We now proceed to show that all undesired summands cancel.
    By truncating the relation \ref{equation_ck^2_QT} to $O(v^2)$ and reordering, one obtains:
    \begin{align*}
        \sum_{i=0}^{2k+2}(-1)^ic^*_{i,2k+2-i}+v_1\sum_{i=0}^{2k+3}(-1)^i(i-1)c^*_{i,2k+3-i}=O(v^2)
    \end{align*}
    Hence, $v_1^3(-T_{3,0}+v_1T_{4,0})$ and $v_2(-T_{0,1}+v_1T_{1,1})$ are both $O(v^5)$ and all “cubic” terms in relation \ref{equation_ck^2_QT} cancel, leaving only the “quadratic” ones. To finish the proof, first truncate to $O(v^4)$ and substitute the expression for $(c^*_{k+1})^2$; one arrives at the result that the desired relation holds up to $O(v^4)$.
    Finally, this allows substituting the expressions for $(c^*_{k+1})^2$ and $(c^*_{k+2})^2$, finishing the proof.
\end{proof}
Now we turn to the relations among the $z_i$. According to Proposition \ref{prop:generators},  they should provide expressions for the squares of these generators in terms of the monomials $z_{1}^{i_1}z_2^{i_2}\dots z_n^{i_n}$, where $i_j\in\{0,1\}$. Below we provide approximations to such relations using the previous proposition and the formulas for $c_i^*$.
\begin{proposition}\label{prop:relations_z_i}
In $\mathrm{BP}^*(\OGr(n))$, the element $(-1)^{k+1}z_k^2$ has the following approximation for $k\geq 3$:
\begin{align*}
\begin{split}
    &z_{2k}+\sum_{i=1}^{k-1}(-1)^i2z_{i,2k-i}+v_1\left(\sum_{i=0}^k(-1)^iR_{1,0}(k,i)z_{i,2k+1-i}\right)+v_1^2\left(\sum_{i=0}^k(-1)^iR_{2,0}(k,i)z_{i,2k+2-i}\right)\\
    +&v_1^3\left(\sum_{i=0}^{k+1}(-1)^iR_{3,0}(k,i)z_{i,2k+3-i}\right)+v_2\left(\sum_{i=0}^{k+1}(-1)^iR_{0,1}(k,i)z_{i,2k+3-i}\right)\\
    +&v_1^4\left(\sum_{i=0}^{k+1}(-1)^iR_{4,0}(k,i)z_{i,2k+4-i}\right)+v_1v_2\left(\sum_{i=0}^{k+1}(-1)^iR_{1,1}(k,i)z_{i,2k+4-i}\right)+O(v^5),
\end{split}
\end{align*}
where $z_0=1$, and the integral-valued functions $R_{s,t}(k,i)=P_{s,t}(k-i)$ for $i>s+3t$, and the exceptional values are as follows:
\setlength{\LTpre}{0pt}
\setlength{\LTpost}{0pt}
\begin{center}
    \renewcommand{\arraystretch}{1.5}
    \begin{longtable}{c|c|c|c|c|c|c}
     & $(1,0)$ & $(2,0)$ & $(3,0)$ & $(0,1)$ & $(4,0)$ & $(1,1)$ \\
    \hline
    $0$ & $k$ & $\tfrac{k^2+k-2}{2}$ & $\tfrac{k^3+3k^2+2k-24}{6}$ & $k-2$ & $\tfrac{k^4+6k^3+23k^2-54k-504}{24}$ & $k^2-21$ \\
    $1$ & $2k-2$ & $k^2-k-2$ & $\tfrac{k^3-k-24}3$ & $2k-6$ & $\tfrac{k^4+2k^3+11k^2-86k-432}{12}$ & $2k^2-4k-40$ \\
    $2$ & \cellcolor{lightgray} & $k^2-2k-1$ & $\tfrac{2k^3-3k^2+k-54}6$ & $2k-8$ & $\tfrac{k^4+11k^2-108k-384}{12}$ & $2k^2-7k-37$ \\
    $3$ & \cellcolor{lightgray} & \cellcolor{lightgray} & $\tfrac{2k^3-9k^2+25k-72}6$ & $2k-10$ & $\tfrac{k^4-4k^3+29k^2-146k-288}{12}$ & $2k^2-11k-28$ \\
    $4$ & \cellcolor{lightgray} & \cellcolor{lightgray} & \cellcolor{lightgray} & \cellcolor{lightgray} & $\tfrac{k^4-8k^3+47k^2-124k-204}{12}$ & $2k^2-8k-22$\\
    \caption{Values of $R_{s,t}(k,i)$ for $i\leq s+3t$ and $3\leq k\leq n$. Columns correspond to pairs $(s,t)$ and the rows to values of $i$.} 
\end{longtable}
\end{center}
In addition, the squares of the first two generators are given by:
\begin{align*}
    \begin{split}
        z_1^2&=z_2+v_1z_3+v_1^2(2z_4-2z_{1,3})+v_1^3(z_5+4z_{1,4}+z_{2,3})+v_2(-z_5+4z_{1,4}+z_{2,3})\\
        &+v_1^4(16z_6-26z_{1,5})+v_1v_2(17z_6-32z_{1,5}+2z_{2,4})+O(v^5),
    \end{split} \\
    \begin{split}
        -z_2^2&=z_4-2z_{1,3}+v_1(2z_5-2z_{1,4}+z_{2,3})+v_1^2(2z_6-z_{2,4})+v_1^3(6z_{1,6}-8z_{2,5}+7z_{3,4})\\&+v_2(2z_{1,6}-4z_{2,5}+6z_{3,4})+v_1^4(7z_8-8z_{1,7}+7z_{2,6}-12z_{3,5})\\&+v_1v_2(13z_8-20z_{1,7}+17z_{2,6}-18z_{3,5})+O(v^5).
    \end{split}
\end{align*}
\end{proposition}
\begin{proof}
    For $k>0$, solving the equation $c^*_k=(-1)^k\sum_{i=0}^\infty d_{i+1}z_{k+i}$ gives:
    \begin{align*}
    \begin{split}
        z_k':=z_k&=(-1)^k\tfrac12c^*_k+\tfrac{v_1}2z_{k+1}-v_1^2z_{k+2}+(4v_1^3+\tfrac72v_2)z_{k+3}-(13v_1^4+15v_1v_2)z_{k+4}+O(v^5).
    \end{split}
    \end{align*}
    For $k\le0$, we take this formula as a (recursive) definition of $z_k'$, where $c_i=0$ for $i<0$.
    Note that $z_0'\ne z_0=1$ even though both are defined, and $z_k'$ are defined and nonzero even for $k<0$ (although it is unclear whether they have a natural interpretation).
    Plugging this expression into the term $(-1)^{k+1}z_k'^2$, expanding using the multinomial formula and the relation involving $(c^*_k)^2$ of the preceding lemma, dividing by 4, and finally expanding using $c^*_k=(-1)^k\sum_{i=0}^\infty d_{i+1}z_{k+i}'$ again one obtains:
    \begin{align*}
    \begin{split}
        (-1)^{k+1}z_k'^2&=\sum_{i=-\infty}^{k-1}2z_{i,2k-i}'+v_1\sum_{i=-\infty}^k(-1)^iP_{1,0}(k-i)z_{i,2k+1-i}'\\
        &+v_1^2\left(\sum_{i=-\infty}^k(-1)^i\left(P_{2,0}(k-i)+\tfrac{9}{4}\cdot2\right)z_{i,2k+2-i}'+(-1)^{k+1}\tfrac94z_{k+1}'^2\right)\\
        &+v_1^3\sum_{i=-\infty}^{k+1}(-1)^i\left(P_{3,0}(k-i)+\tfrac{9}{4}P_{1,0}(k+1-i)\right)z_{i,2k+3-i}'\\
        &+v_2\sum_{i=-\infty}^{k+1}(-1)^iP_{0,1}(k-i)z_{i,2k+3-i}'\\
        &+v_1^4\left(\sum_{i=-\infty}^{k+1}(-1)^i\left(P_{4,0}(k-i)+\tfrac{9}{4}P_{2,0}(k+1-i)+36\cdot2\right)z_{i,2k+4-i}'+(-1)^{k+2}36z_{k+2}'^2\right)\\
        &+v_1v_2\left(\sum_{i=-\infty}^{k+1}(-1)^i\left(P_{1,1}(k-i)+\tfrac{75}{2}\cdot2\right)z_{i,2k+4-i}'+(-1)^{k+2}\tfrac{75}2z_{k+2}'^2\right)+O(v^5).
    \end{split}
    \end{align*}
    Using the analogous strategy of successive truncation and “refinement” by substitution of $z_{k+1}'^2$ and $z_{k+2}'^2$ as in the proof of the preceding lemma, one arrives at the following relation that does not contain squares any more:
    \begin{align*}
    \begin{split}
        (-1)^{k+1}z_k'^2&=\sum_{i=-\infty}^{k-1}2z_{i,2k-i}'+v_1\sum_{i=-\infty}^k(-1)^iP_{1,0}(k-i)z_{i,2k+1-i}'+v_1^2\sum_{i=-\infty}^k(-1)^iP_{2,0}(k-i)z_{i,2k+2-i}'\\
        &+v_1^3\sum_{i=-\infty}^{k+1}(-1)^iP_{3,0}(k-i)z_{i,2k+3-i}'+v_2\sum_{i=-\infty}^{k+1}(-1)^iP_{0,1}(k-i)z_{i,2k+3-i}'\\
        &+v_1^4\sum_{i=-\infty}^{k+1}(-1)^iP_{4,0}(k-i)z_{i,2k+4-i}'+v_1v_2\sum_{i=-\infty}^{k+1}(-1)^iP_{1,1}(k-i)z_{i,2k+4-i}'+O(v^5).
    \end{split}
    \end{align*}
    Note that this formula holds for all positive $k$ including $k=1,2$. We now eliminate the infinite sums. A straightforward computation from the definition shows:
    \begin{align*}
        z_0'&=\tfrac{1}{2}+\tfrac{v_1}{2}z_1-v_1^2z_2+(4v_1^3+\tfrac{7}{2}v_2)z_3-(13v_1^4+15v_1v_2)z_4+O(v^5),\\
        z_{-1}'&=\tfrac{v_1}{4}-\tfrac{3}{4}v_1^2z_1+(\tfrac{7}{2}v_1^3+\tfrac{7}{2}v_2)z_2-(11v_1^4+\tfrac{53}{4}v_1v_2)z_3+O(v^5),\\
        z_{-2}'&=-\tfrac{3}{8}v_1^2+(\tfrac{25}{8}v_1^3+\tfrac{7}{2}v_2)z_1-(\tfrac{41}{4}v_1^4+\tfrac{53}{4}v_1v_2)z_2+O(v^5),\\
        z_{-3}'&=(\tfrac{25}{16}v_1^3+\tfrac{7}{4}v_2)-(\tfrac{139}{16}v_1^4+\tfrac{23}{2}v_1v_2)z_1+O(v^5),\\
        z_{-4}'&=-(\tfrac{139}{32}v_1^4+\tfrac{23}{4}v_1v_2)+O(v^5),\\
        z_i'&=O(v^5)\ \text{ for }\ i\le-5.
    \end{align*}
    Substituting these relations one arrives the desired result for $k\ge3$ immediately.

    For $k=1,2$, starting from $O(v^2)$ there appear summands $z_i^2$ for $2\le i\le4$.
    Using the known relations and the refinement argument once more, the claimed result follows here as well.
\end{proof}

\section{Search for irrational \texorpdfstring{$v_2$}{v2}-torsion}\label{section:6}
Throughout this section $\tilde{q}_{\mathrm{gen}}$ is a generic quadratic form with trivial discriminant and trivial Clifford invariant of dimension $2n+1$. In other words, this is a quadratic form that corresponds to a generic $\mathrm{Spin}_{2n+1}$-torsor.

\begin{definition}
    Let $A^*$ be an oriented cohomology theory satisfying the global prerequisites of \cite{czz-equivariant}; for our purposes it is enough to know that $A^*$ being a free theory with $2\in A^*(F)$ regular is sufficient.
    Let again $x_i$ denote the Chern roots of the vector bundle $\taut_n^\vee$ over $\OGr(n)$.
    As in \cite[\S 8.2]{km-canonical-p-dimension}, denote by $y$ the character satisfying $2y=x_1+\ldots+x_n$.
    By \cite[Theorem 9.1]{czz-equivariant} and the obvious symmetry, the Chern class $c_1^A(y)\in A^1(\mathrm{Spin}_{2n+1}/B)$ is contained in the image of the canonical injection (pullback of the projection) $A^*(\OGr(n))\rightarrow A^*(\mathrm{Spin}_{2n+1}/B)$.
    Define $u\in A^1(\OGr(n))$ as its unique preimage.
\end{definition}

Directly from the definition it is possible to compute the following approximation of the element $u$.
We will need this in our final computations.
\begin{lemma}
    The following identity holds in $\mathrm{BP}^1(\OGr(n))$ (and consequently, $\BPtr^1(\OGr(n))$):
    \begin{align*}
        u&=-z_1-v_1^2z_{1,2}+v_1^3(5z_4-z_{1,3})+v_2(4z_4-z_{1,3})+v_1^4(-4z_5-6z_{1,4}-z_{2,3})\\
        &+v_1v_2(-6z_5-8z_{1,4}+z_{2,3})+O(v^5).
    \end{align*}
\end{lemma}
\begin{proof}
    Denote by $F$ the formal group law of $\mathrm{BP}$ for brevity. It is immediate from the definition of $u\in\mathrm{BP}^1(\OGr(n))$ that $[2]_F(u)=c_1(x_1)+_F\ldots+_Fc_1(x_n)$.
    In other words, if
    $$l(t)=\sum_{i=1}^\infty l_it^i\in\mathbb{Q}[v_1,v_2,\dots][[t]]$$
    is the logarithm of the formal group law of $\mathrm{BP}^*$, we have:
    \begin{align*}
        \begin{split}
            u&=l^{-1}\left(\tfrac{l(c_1(x_1))+\ldots+l(c_1(x_n))}2\right)=l^{-1}\left(\tfrac12\sum_{i=1}^\infty l_i\sum_{j=1}^nc_1(x_j)^i\right)\\
            &=l^{-1}\left(\tfrac12(l_1p_1(c_1(x_1),\ldots,c_1(x_n))+\ldots+l_5p_5(c_1(x_1),\ldots,c_1(x_n)))\right)+O(v^5),
        \end{split}
    \end{align*}
    where $p_i$'s are the power-sum symmetric polynomials.
    Expressing these symmetric polynomials in terms of the elementary symmetric polynomials we obtain
    \begin{align*}
        u=l^{-1}(\tfrac12(l_1c_1^*+\ldots+l_5((c_1^*)^5-5(c_1^*)^4c_2^*+5c_1^*(c_2^*)^2+5(c_1^*)^2c_3^*-5c_2^*c_3^*-5c_1^*c_4^*+5c_5^*))+O(v^5).
    \end{align*}
    Plugging in the explicit expressions of the logarithm and the exponential (see Appendix \ref{appendix_fgl}), and simplifying the terms fully using Proposition \ref{prop:generators} and Proposition \ref{prop:relations_z_i}, one arrives at the given result.
    Explicitly, the simplifications are performed using the computer for $n=5$ (see \S \ref{subsection_code} for the link to the code and output files), and one observes that the resulting expression has to look the same for all $n\ge5$.
\end{proof}

\begin{lemma}\label{lemma:rat_generators}
    Let $A^*$ be a free oriented cohomology theory and let $n$ be a natural number. Then the $A^*(F)$-submodule of rational elements $\overline{A}^*(\OGr(\tilde{q}_{\mathrm{gen}}))\subset A^*(\OGr(n))$ is generated by the elements 
    $$ u^k(c^*_{2})^{i_2}(c^*_{3})^{i_3}\dots(c^*_{n})^{i_n}, $$
    where the indexes run over $0\leq i_j\leq 1$ and $0\leq k\leq \mathrm{dim}(\OGr(n))$.
\end{lemma}
\begin{proof}
    All elements above are clearly rational. We claim that their arbitrary lifts generate the $A^*(F)$-module $A^*(\OGr(\tilde{q}_{\mathrm{gen}}))$. To prove this, it is sufficient to treat the case of algebraic cobordism $A^*=\Omega^*$. Moreover, it is well known that the images of (lifts of) these monomials under the projection $\Omega^*\to \CH^*$ generate $\CH^*(\OGr(\tilde{q}_{\mathrm{gen}}))$, see e.g., \cite[Proposition 2.1]{Karpenko_Spin_11_12} (note that for $A^*=\CH^*$ we have $u=-z_1$). The result then follows from the graded Nakayama lemma.
\end{proof}

In order to simplify computations with irrational elements, we need the following folklore fact, for which we did not find a reference in such generality.
\begin{proposition}\label{prop:t_mult_rat}
    Let $A^*$ be a free oriented cohomology theory.
    Denote by $t$ the torsion index of $\mathrm{Spin}_{2n+1}$.
    Then the $A^*(F)$-submodule of rational elements $\overline{A}^*(\OGr(\tilde{q}_{\mathrm{gen}}))$ contains $t\cdot A^*(\OGr(n))$.
    In other words, any multiple of $t$ is rational.
\end{proposition}
\begin{proof}
    It suffices to consider the case of algebraic cobordism $A^*=\Omega^*$, as any other free theory is obtained by tensor product from it. Indeed, suppose we have an inclusion $t\cdot \Omega^*(\OGr(n))\subset \overline{\Omega}^*(\OGr(\tilde{q}_{\mathrm{gen}}))$ of subgroups of $\Omega^*(\OGr(n))$. The tensor product $-\otimes_{\mathbb{L}} A^*(F)$ and compatibility of $\Omega^*\to A^*$ with restrictions yield maps $$t\cdot A^*(\OGr(n))\to \overline{\Omega}^*(\OGr(\tilde{q}_{\mathrm{gen}}))\otimes_{\mathbb{L}}A^*(F)\to \overline{A}^*(\OGr(\tilde{q}_{\mathrm{gen}})).$$
    Since the composite is compatible with the obvious inclusions as subsets of $A^*(\OGr(n))$, it must be injective.
    Consequently, $t\cdot A^*(\OGr(n))$ is a subgroup of $\overline{A}^*(\OGr(\tilde{q}_{\mathrm{gen}}))$ as required.
    
    Let $L/F$ be a finite field extension that completely splits $\tilde{q}_{\mathrm{gen}}$. It follows that $\Omega^*(\OGr(\tilde{q}_{\mathrm{gen}})_L) = \Omega^*(\OGr(n))$. We claim that 
    $$\Omega^*(\OGr(n))\xrightarrow{N_{L/F}^\Omega}\Omega^*(\OGr(\tilde{q}_{\mathrm{gen}}))\xrightarrow{\res^\Omega_{L/F}}\Omega^*(\OGr(n))$$
    is multiplication by $[L:F]$, where $N_{L/F}^\Omega$ is the norm map (pushforward along the base change).
    
    Let $x\in\Omega^*(\OGr(n))$.
    We have isomorphisms of presheaves of rings:
    $$ \Omega^*\otimes \mathbb{Q}\xrightarrow{\simeq} \CH^*[b_1,b_2,\dots]\otimes \mathbb{Q}\xleftarrow{\simeq} \mathrm{K}_0[b_1,b_2,\dots]\otimes \mathbb{Q},  $$
    where the first isomorphism is induced by the logarithm of the universal formal group law (see \cite[Theorem 4.1.28]{LevMor}), and the second isomorphism is induced by the Chern character. Moreover, by Panin's result \cite{Panin_E/P}, the restriction map $\res_{L/F}^{\mathrm{K}_0}$ for $\OGr(\tilde{q}_{\mathrm{gen}})$ is an isomorphism. It follows that the restriction map $\res_{L/F}^{\Omega\otimes \mathbb{Q}}$ is also bijective. Hence, there exist $z\in\Omega^*(\OGr(\tilde{q}_{\mathrm{gen}}))$ and $n\in\mathbb{Z}$ such that $nx=\res^{\Omega}_{L/F}(z)$. Now consider the following chain of equalities:
    \begin{align*}
        n(\res_{L/F}^\Omega \circ N_{L/F}^\Omega)(x)=(\res_{L/F}^\Omega\circ N_{L/F}^\Omega\circ \res_{L/F}^\Omega)(z)= \res_{L/F}^\Omega ([L:F] z)=n[L:F]x.
    \end{align*}
    where the second one follows from the projection formula (the norm of $1$ is equal to $[L:F]$ by \cite[Lemma 2.3.5(1)]{LevMor}). The claim follows dividing the left-hand and right-hand sides by $n$, which is possible since $\Omega^*(\OGr(n))$ is torsion-free.
    
    The claim implies the result. Indeed, for the quadratic form $\tilde{q}_{\mathrm{gen}}$ the index $\mathrm{ind}(\OGr(\tilde{q}_{\mathrm{gen}}))$ (which is equal to the greatest common divisor of degrees of extensions $L/F$ that split $\tilde{q}_{\mathrm{gen}}$) coincides with the torsion index $t$ by Grothendieck's theorem \cite[Th\'eor\`eme 2]{Grothendieck}.
\end{proof}

\begin{remark}
    The above proposition holds for any projective homogeneous variety $X$ under a semisimple group of inner type $G$. The last assumption is needed for the result from \cite{Panin_E/P}, which provides an isomorphism $\res_{L/F}^{\mathrm{K}_0\otimes \mathbb{Q}}\colon\mathrm{K}_0(X)\otimes \mathbb{Q}\xrightarrow{\simeq} \mathrm{K}_0(X_L)\otimes \mathbb{Q}$ for any splitting field $L/F$.
\end{remark}

\begin{proposition}\label{prop:desired_v2_tors}
    The following elements are not rational in $\mathrm{BP}\langle2\rangle^*(\OGr(n))$, but their multiples by $v_2$ are:
    \begin{align*}
        4v_1z_{1,2,3,4,5,6,7}&\in\mathrm{BP}\langle2\rangle^*(\OGr(7)),\\
        8z_{1,2,3,4,5,6,7,8}&\in\mathrm{BP}\langle2\rangle^*(\OGr(8)),\\
        8z_{1,2,3,4,5,6,7,8,9}&\in\mathrm{BP}\langle2\rangle^*(\OGr(9)),\\
        16z_{1,2,3,4,5,6,7,8,9,10}&\in\mathrm{BP}\langle2\rangle^*(\OGr(10)),\\
        64z_{1,2,3,4,5,6,7,8,9,10,11,12,13}&\in\mathrm{BP}\langle2\rangle^*(\OGr(13)).
    \end{align*}
\end{proposition}
\begin{proof}
    By Proposition \ref{prop:t_mult_rat} it is enough to perform all computations modulo the torsion index, which is determined by \cite[Theorem 0.1]{totaro} as $t=2^{u(n)}$ for an explicitly given number $u(n)$ (not to be confused with our element $u$).
    Concretely, using a computer (see \S \ref{subsection_code}, here we only need the approximate computations) one expresses all elements of the form $u^kc^*_I$ in terms of the $z_I$, to obtain:
    \begin{align*}
        u^{15}c^*_{4,5}&\equiv4v_1v_2z_{1,2,3,4,5,6,7}+O(v^5)\mod8\mathrm{BP}\langle2\rangle^*(\OGr(7)),\\
        u^{15}c^*_{2,3,6,7}+2u^{23}c^*_{4,6}&\equiv8v_2z_{1,2,3,4,5,6,7,8}+O(v^5)\mod16\mathrm{BP}\langle2\rangle^*(\OGr(8)),\\
        u^{31}c^*_{2,3,6}+v_1u^{31}c^*_{2,4,6}&\equiv8v_2z_{1,2,3,4,5,6,7,8,9}+O(v^5)\mod16\mathrm{BP}\langle2\rangle^*(\OGr(9)),\\
        u^{31}c^*_{4,8,9}&\equiv16v_2z_{1,2,3,4,5,6,7,8,9,10}+O(v^5)\mod32\mathrm{BP}\langle2\rangle^*(\OGr(10)),\\
        u^{63}c^*_{4,9,12}&\equiv64v_2z_{1,2,3,4,5,6,7,8,9,10,11,12,13}+O(v^5)\mod128\mathrm{BP}\langle2\rangle^*(\OGr(13)).
    \end{align*}
    For dimensional reasons, the $O(v^5)$ terms must be 0, so the equalities hold exactly.

    On the other hand, in the case $n=7$ none of the computed expressions contains a term dividing $4z_{1,2,3,4,5,6,7}$ or $4v_1z_{1,2,3,4,5,6,7}$, concluding the proof for this $n$.
    In the other cases, it is already enough to notice that if $2^{u(n)-1}z_{1,\ldots,n}\in\mathrm{BP}\langle2\rangle^*(\OGr(n))$ were rational, the same would be the case in $\CH^*(\OGr(n))\otimes\mathbb{Z}_{(2)}$, contradicting Grothendieck's theorem \cite[Th\'eor\`eme 2]{Grothendieck}.
\end{proof}

The cases $n=14,15,16$ in the main theorem will be concluded by induction from $n=13$ (the same argument would also work for going from $n=9$ to $n=10$, but the counter-example presented above looks slightly nicer).
For this purpose, we also need the following lemma.

\begin{lemma}\label{lemma:induction}
    Let $A^*$ be a free oriented cohomology theory such that $2\in A^*(F)$ is regular. Assume that $mz_{1,\ldots,n}\in A^*(\OGr(n))$ is rational for some $m\in\mathbb{Z}$.
    Then $2mz_{1,\ldots,n+1}\in A^*(\OGr(n+1))$ is rational.
\end{lemma}
\begin{proof}
    Using the expression $[2]_F(u)=c_1(x_1)+_F\ldots+_Fc_1(x_{n+1})$ and the same approach of Lemma \ref{pullback1}, one sees that $j_{n+1}^*(u)=u\in A^1(\OGr(n))$.
    By Lemma \ref{lemma:rat_generators} we have $mz_{1,\ldots,n}=\sum_{k,I}a_{k,I}u^kc^*_I$ for some (not necessarily unique, so we fix one choice) $a_{k,I}\in A^*(F)$. The following vanishing is then obvious:
    \begin{align*}
        j^*_{n+1}\left(\sum_{k,I}a_{k,I}u^kc^*_I-mz_{1,\ldots,n}\right)&=0.
    \end{align*}
    By Lemma \ref{pullback1} again, the kernel of $j_{n+1}^*$ is generated as an ideal by $z_{n+1}$. Therefore, we obtain:
    \begin{align*}
        \sum_{k,I}a_{k,I}u^kc^*_{I\cup\{n+1\}}-mz_{1,\ldots,n}c^*_{n+1}&\in(z_{n+1}c^*_{n+1})\subseteq A^*(\OGr(n+1)).
    \end{align*}
    As $(c^*_{n+1})^2=0$ and $c^*_{n+1}=(-1)^{n+1}2z_{n+1}$, we have $z_{n+1}c^*_{n+1}=0$. Combining this observation with the previous formula we get
    \begin{align*}
        0&=\sum_{k,I}a_{k,I}u^kc^*_{I\cup\{n+1\}}-mz_{1,\ldots,n}c^*_{n+1}=\sum_{k,I}a_{k,I}u^kc^*_{I\cup\{n+1\}}-(-1)^{n+1}2mz_{1,\ldots,n+1}.
    \end{align*}
    This gives us an expression of $2mz_{1,\dots,n+1}$ as a linear combination of rational elements as required.
\end{proof}

\begin{proof}[Proof of Theorem \ref{main_theorem}]
    By \cite[Proposition 2.16]{Baek_Karpenko_counter-example} it suffices to consider the case of $\mathrm{Spin}_{2n+1}$. For $n=7,8,9,10,13$, Proposition \ref{prop:desired_v2_tors} provides a non-trivial element in the Tor group (see Example \ref{example:bptr})
    $$\Tor_1^{\mathbb{Z}_{(2)}[v_1,v_2]}(\underline{\BPtr}^*(\OGr(\tilde{q}_{\mathrm{gen}})),\mathbb{Z}_{(2)}[\beta]),$$
    and the previous lemma allows to conclude the same by induction for $n=14,15,16$ since in these cases $u(n)=u(n-1)+1$.
    By Proposition \ref{prop:p_local_karpenko}, the existence of such an element contradicts the conjecture.
\end{proof}

\begin{remark}\label{remark_degree_chow}
    By the proof of Proposition \ref{prop:p_local_karpenko} we construct non-trivial elements in the kernel of $$\overline{\BPtr}^*(\OGr(\tilde{q}_{\mathrm{gen}}))\otimes_{\mathbb{Z}_{(2)}[v_1,v_2]}\mathbb{Z}_{(2)}[\beta]\to \CK^*(\OGr(n)),$$
    which are given by $v_2$-multiples of elements stated in Proposition \ref{prop:desired_v2_tors} (note that those $v_2$-multiplies are non-trivial in the tensor product since we cannot move $v_2$ to the second factor). Lifts of these elements to $\CK^*(\OGr(\tilde{q}_{\mathrm{gen}}))$ through the epimorphism
    $$\CK^*(\OGr(\tilde{q}_{\mathrm{gen}}))\twoheadrightarrow \overline{\BPtr}^*(\OGr(\tilde{q}_{\mathrm{gen}}))\otimes_{\mathbb{Z}_{(2)}[v_1,v_2]}\mathbb{Z}_{(2)}[\beta]$$
    give torsion elements in $\CK^*(\OGr(\tilde{q}_{\mathrm{gen}}))$. Degrees of those are $\mathrm{dim}(\OGr(\tilde{q}_{\mathrm{gen}}))-3$ in all cases except $n=7$, when it is given by $\mathrm{dim}(\OGr(\tilde{q}_{\mathrm{gen}}))-4$. We claim that the resulting elements are not divisible by $\beta$. Indeed, for $n=7$, if it were divisible by $\beta$, then $$4v_1v_2z_{1,2,3,4,5,6,7}\otimes 1\in \overline{\BPtr}^*(\OGr(\tilde{q}_{\mathrm{gen}}))\otimes_{\mathbb{Z}_{(2)}[v_1,v_2]}\mathbb{Z}_{(2)}[\beta]$$ would be divisible by $\beta$, which is impossible as $4v_2z_{1,2,3,4,5,6,7}$ is not rational and $\overline{\BPtr}^*(\OGr(\tilde{q}_{\mathrm{gen}}))$ has no $v_2$-torsion.
    In the remaining cases, the corresponding $v_2$-multiples clearly cannot be divided by $v_1$ even in $\BPtr^*(\OGr(n))$.
    
    It follows from the claim and proof of \cite[Theorem 3.1]{KarpCK} that these elements yield us torsion elements of the same degrees in $\CH^*(\OGr(\tilde{q}_{\mathrm{gen}}))$ that vanish under $\varphi_{\OGr(\tilde{q}_{\mathrm{gen}})}$. In particular, for $n=7$ the constructed torsion element belongs to $\CH_4(\OGr(\tilde{q}_{\mathrm{gen}}))$, while for all other $n$ the corresponding elements belong to  $\CH_3(\OGr(\tilde{q}_{\mathrm{gen}}))$.
\end{remark}

\appendix

\section{The universal \texorpdfstring{$p$}{p}-typical formal group law}\label{appendix_fgl}
The formal group law of the Brown-Peterson cohomology is known more or less explicitly, in terms of its logarithm.
More precisely, writing it as $l(t)=\sum_{i\ge0}l_it^{p^i}\in\mathbb{Z}_{(p)}[v_1,v_2,\ldots][[t]]$, one choice is given by Hazewinkel's recursive formula following \cite[A2.2.1]{Ra04}
\begin{align*}
    l_0&=1, \\
    l_n&=p^{-1}\sum_{i=0}^{n-1}l_iv_{n-i}^{p^i}.
\end{align*}
In concrete terms, for $p=2$ the power series and its compositional inverse (computed by hand or using Sage) are as follows:
\begin{align*}
    l(t)&=t+\tfrac12v_1t^2+\left(\tfrac14v_1^3+\tfrac12v_2\right)t^4+O(v^5), \\
    l^{-1}(t)&=t-\tfrac12v_1t^2+\tfrac12v_1^2t^3-\left(\tfrac78v_1^3+\tfrac12v_2\right)t^4+\left(\tfrac{13}8v_1^4+\tfrac32v_1v_2\right)t^5+O(v^5).
\end{align*}
The approximation of the formal group law $F_{\mathrm{BP}}$ now can be computed easily
\begin{align*}
\begin{split}
    x+_F y=l^{-1}(l(x)+l(y))&=x+y-v_1xy+v_1^2(x^2y+xy^2)-v_1^3(2x^3y+4x^2y^2+2xy^3)\\
    &-v_2(2x^3y+3x^2y^2+2xy^3)+v_1^4(3x^4y+10x^3y^2+10x^2y^3+3xy^4)\\
    &+v_1v_2(4x^4y+11x^3y^2+11x^2y^3+4xy^4)+O(v^5).
\end{split}
\end{align*}
In the main part of the text we also use the approximations of the inverse power series and the formal multiplication by $2$
\begin{align*}
    [-1]_F(t)&=l^{-1}(-l(t))=-t-v_1t^2-v_1^2t^3-(2v_1^3+v_2)t^4-(4v_1^4+3v_1v_2)t^5+O(v^5), \\
    [2]_F(t)&=l^{-1}(2l(t))=2t-v_1t^2+2v_1^2t^3-(8v_1^3+7v_2)t^4+(26v_1^4+30v_1v_2)t^5+O(v^5).
\end{align*}

\section{Some symmetric polynomial identities}\label{appendix_sym_pol}
Consider the symmetric polynomial ring $R=\mathbb{Z}[x_1,\dots,x_n]^{S_n}$. Denote by $\sigma_k\in R$ the $k$-th elementary symmetric polynomial. It is well-known that they freely generate $R$ as a ring. For an integer partition $\lambda=(\lambda_1,\dots,\lambda_m)$ we put $\sigma_{\lambda}=\prod_{i=1}^m \sigma_{\lambda_i}\in R$. Moreover, for $\lambda$ we also introduce the monomial symmetric polynomial $m_\lambda=\sum_{g\in S_n/(S_n)_\lambda} x^{g\lambda}\in R$. We also denote by $\lambda^k$ the partition $(\lambda,\dots,\lambda)$ of $k\lambda$.

Our goal is to express certain monomial symmetric polynomials as linear combinations of $\sigma_\lambda$'s.
The following identities expressing a product of at most two elementary or monomial symmetric polynomials as a linear combination of monomial symmetric polynomials will be useful, and are easily obtained from the definitions:
\begin{align*}
        \sigma_i&=m_{1^i}\\
        \sigma_{1,i}&=m_{2,1^{i-1}}+(i+1)m_{1^{i+1}}\\
        \sigma_{2,i}&=m_{2^2,1^{i-2}}+im_{2,1^i}+\binom{i+2}2m_{1^{i+2}}\\
        \sigma_{3,i}&=m_{2^3,1^{i-3}}+(i-1)m_{2^2,1^{i-1}}+\binom{i+1}2m_{2,1^{i+1}}+\binom{i+3}3m_{1^{i+3}}\\
        \sigma_{4,i}&=m_{2^4,1^{i-4}}+(i-2)m_{2^3,1^{i-2}}+\binom i2m_{2^2,1^i}+\binom{i+2}3m_{2,1^{i+2}}+\binom{i+4}4m_{1^{i+4}}\\
        \sigma_1m_{2,1^{i-1}}&=m_{3,1^{i-1}}+2m_{2^2,1^{i-2}}+im_{2,1^i}\\
        \sigma_2m_{2,1^{i-1}}&=m_{3,2,1^{i-2}}+im_{3,1^i}+3m_{2^3,1^{i-3}}+2(i-1)m_{2^2,1^{i-1}}+\binom{i+1}2m_{2,1^{i+1}}\\
        \begin{split}
            \sigma_3m_{2,1^{i-1}}&=m_{3,2^2,1^{i-3}}+(i-1)m_{3,2,1^{i-1}}+\binom{i+1}2m_{3,1^{i+1}}+4m_{2^4,1^{i-4}}+3(i-2)m_{2^3,1^{i-2}}\\
            &+2\binom i2m_{2^2,1^i}+\binom{i+2}3m_{2,1^{i+2}}
        \end{split} \\
        \sigma_1m_{3,1^{i-1}}&=m_{4,1^{i-1}}+m_{3,2,1^{i-2}}+im_{3,1^i}\\
        \sigma_2m_{3,1^{i-1}}&=m_{4,2,1^{i-2}}+im_{4,1^i}+m_{3,2^2,1^{i-3}}+(i-1)m_{3,2,1^{i-1}}+\binom{i+1}2m_{3,1^{i+1}}\\
        \sigma_2m_{2^2,1^{i+2}}&=m_{3^2,1^{i-2}}+2m_{3,2^2,1^{i-3}}+(i-1)m_{3,2,1^{i-1}}+6m_{2^4,1^{i-4}}+3(i-2)m_{2^3,1^{i-2}}+\binom i2m_{2^2,1^i}\\
        \sigma_1m_{4,1^{i-1}}&=m_{5,1^{i-1}}+m_{4,2,1^{i-2}}+im_{4,1^i}
    \end{align*}
    We can now solve these equations for the monomial symmetric polynomials, by proceeding in the graded-lexicographic order.
    The general idea is to start by decomposing the leading monomial in a way that the exponents of both factors stay nonincreasing, and then look at the product of the corresponding monomial symmetric polynomials.
    We already know how to express both factors in terms of elementary symmetric polynomials, and the same holds for the undesired monomials appearing in the product, since by construction the leading monomial of the product is the correct one.
    The following expressions are obtained (the two factors are always the unique ones in the above relations giving the correct leading monomial):
    \begin{align*}
        m_{1^i}&=\sigma_i\\
        m_{2,1^{i-1}}&=\sigma_{1,i}-(i+1)\sigma_{i+1}\\
        m_{2^2,1^{i-2}}&=\sigma_{2,i}-i\sigma_{1,i+1}+\tfrac{i^2+i-2}2\sigma_{i+2}\\
        m_{3,1^{i-1}}&=\sigma_{1^2,i}-2\sigma_{2,i}-\sigma_{1,i+1}+(i+2)\sigma_{i+2}\\
        m_{2^3,1^{i-3}}&=\sigma_{3,i}-(i-1)\sigma_{2,i+1}+\tfrac{i^2-i-2}2\sigma_{1,i+2}-\tfrac{i^3-7i+6}6\sigma_{i+3}\\
        m_{3,2,1^{i-2}}&=\sigma_{1,2,i}-3\sigma_{3,i}+(2i-2)\sigma_{2,i+1}-i\sigma_{1^2,i+1}+(2i+1)\sigma_{1,i+2}-(i^2+2i-3)\sigma_{i+3}\\
        m_{4,1^{i-1}}&=\sigma_{1^3,i}-3\sigma_{1,2,i}+3\sigma_{3,i}-\sigma_{1^2,i+1}+2\sigma_{2,i+1}+\sigma_{1,i+2}-(i+3)\sigma_{i+3}\\
        \begin{split}
            m_{2^4,1^{i-4}}&=\sigma_{4,i}-(i-2)\sigma_{3,i+1}+\tfrac{i^2-3i}2\sigma_{2,i+2}-\tfrac{i^3-3i^2-4i+12}6\sigma_{1,i+3}\\
            &+\tfrac{i^4-2i^3-13i^2+38i-24}{24}\sigma_{i+4}
        \end{split} \\
        \begin{split}
            m_{3,2^2,1^{i-3}}&=\sigma_{1,3,i}-4\sigma_{4,i}+(3i-6)\sigma_{3,i+1}-(i-1)\sigma_{1,2,i+1}-(i^2-4i)\sigma_{2,i+2}+\tfrac{i^2-i-2}2\sigma_{1^2,i+2}\\
            &-\tfrac{3i^2-i-10}2\sigma_{1,i+3}+\tfrac{i^3+i^2-10i+8}2\sigma_{i+4}
        \end{split} \\
        \begin{split}
            m_{3^2,1^{i-2}}&=\sigma_{2^2,i}-2\sigma_{1,3,i}+2\sigma_{4,i}-\sigma_{1,2,i+1}+3\sigma_{3,i+1}-(2i+1)\sigma_{2,i+2}+(i+1)\sigma_{1^2,i+2}\\
            &-(i+1)\sigma_{1,i+3}+\tfrac{i^2+3i-4}2\sigma_{i+4}
        \end{split} \\
        \begin{split}
            m_{4,2,1^{i-2}}&=\sigma_{1^2,2,i}-2\sigma_{2^2,i}-\sigma_{1,3,i}+4\sigma_{4,i}+(3i-1)\sigma_{1,2,i+1}-i\sigma_{1^3,i+1}-(3i-3)\sigma_{3,i+1}\\
            &-(2i-2)\sigma_{2,i+2}+i\sigma_{1^2,i+2}-(2i+2)\sigma_{1,i+3}+(i^2+3i-4)\sigma_{i+4}
        \end{split} \\
        \begin{split}
            m_{5,1^{i-1}}&=\sigma_{1^4,i}-4\sigma_{1^2,2,i}+4\sigma_{1,3,i}+2\sigma_{2^2,i}-4\sigma_{4,i}-\sigma_{1^3,i+1}+3\sigma_{1,2,i+1}-3\sigma_{3,i+1}+\sigma_{1^2,i+2}\\
            &-2\sigma_{2,i+2}-\sigma_{1,i+3}+(i+4)\sigma_{i+4}
        \end{split}
    \end{align*}

\bibliographystyle{siam}
\bibliography{Karpenko_spin}

\end{document}